\documentclass[a4paper,12pt]{amsart}

\makeindex

\usepackage{xcolor}
\usepackage[
colorlinks=true,
urlcolor=purple,
linkcolor=purple!87!black,
pdfborder={0 0 0}
]{hyperref}

\usepackage{amsfonts,graphics,amsmath,amsthm,amsfonts,amscd,amssymb,amsmath,latexsym,euscript, enumerate,kotex}
\usepackage{epsfig,url}
\usepackage{flafter}
\usepackage[all,cmtip,line]{xy}
\usepackage{array}
\usepackage[english]{babel}
\usepackage{overpic}
\usepackage{subfig}
\usepackage{multirow}
\usepackage{tikz-cd}
\usepackage{theoremref}

\usepackage{wrapfig}
\usepackage[shortlabels]{enumitem}
\setlist[enumerate, 1]{1\textsuperscript{o}}

\newtheorem{thm}{Theorem}[section]
\newtheorem{lem}[thm]{Lemma}

\newtheorem{coro}[thm]{Corollary}

\theoremstyle{definition} 
\newtheorem{defi}[thm]{Definition}
\newtheorem{definition-lemma}[thm]{Definition-Lemma}

\theoremstyle{remark}
\newtheorem{rmk}[thm]{Remark}

\newtheorem*{ack}{Acknowledgments}

\numberwithin{equation}{section}

\newcommand{\C}{\mathbb{C}}

\newcommand{\Q}{\mathbb{Q}}

\def\P{\mathbb{P}}

\DeclareRobustCommand{\O}{\mathcal{O}}

\newcommand{\floor}[1]{\left\lfloor #1 \right\rfloor}

\let\oldframe\frame
\renewcommand\frame[1][allowframebreaks]{\oldframe[#1]}

\title[The number of smooth varieties in an MMP on a 3-fold of Fano type]
{The number of smooth varieties in an MMP on a 3-fold of Fano type}
\begin{document}

\author{Donghyeon Kim}
\address{Department of mathematics, Yonsei University, 50 Yonsei-Ro, Seodaemun-Gu, Seoul 03722, Korea}
\email{narimial0@gmail.com (primary), whatisthat@yonsei.ac.kr (secondary)}

\date{\today}

\begin{abstract}
In this paper, we prove that for a threefold of Fano type $X$ and a movable $\mathbb{Q}$-Cartier Weil divisor $D$ on $X$, the number of smooth varieties that arise during the running of a $D$-MMP is bounded by $1 + h^1(X,2D)$. Additionally, we prove a partial converse to the Kodaira vanishing theorem for a movable divisor on a threefold of Fano type.
\end{abstract}

\maketitle

\section{Introduction}
It is well known that, unlike the case of surfaces, the outcomes of the minimal model program in the case of threefolds can result in singular varieties. This motivates the need to define and study singularities such as terminal singularities, kawamata log terminal singularities, and log canonical singularities in the minimal model program.

It is natural to ask how frequently the varieties produced by running a minimal model program are singular. We will particularly examine how many smooth varieties can be generated during the minimal model program for a threefold of Fano type.

Let $X$ be a threefold of Fano type and let $D$ be a movable $\Q$-Cartier Weil divisor on $X$. If $D=-K_X$, then \cite[Theorem 0]{Ben85} tells us that in a $-K_X$-MMP, there are no smooth varieties except to $X$ itself. The following theorem addresses the problem in the case of a $D$-MMP for a general $D$.

\begin{thm}\thlabel{main}
Let $X$ be a threefold of Fano type, and suppose that $D$ is a movable $\Q$-Cartier Weil divisor on $X$. Let
$$ X_0:=X\overset{f_1}{\dashrightarrow} X_1\overset{f_2}{\dashrightarrow} \cdots \overset{f_n}{\dashrightarrow} X_n$$
be a $D$ -MMP, $D_i$ the strict transform of $D$ to $X_i$, $m$ a positive integer, and let us denote by $N$ the number of $0\le i\le n$ such that $mD_i$ is Cartier. Then we have
$$ N\le 1+h^1(X,2mD).$$
\end{thm}

Note that if we put $m=1$, then we obtain the following.

\begin{coro}\thlabel{minicoro}
Let $X$ be a threefold of Fano type, and suppose that $D$ is a movable $\Q$-Cartier Weil divisor on $X$. Let
$$ X_0:=X\overset{f_1}{\dashrightarrow} X_1\overset{f_2}{\dashrightarrow} \cdots \overset{f_n}{\dashrightarrow} X_n$$
be a $D$-MMP. Then the number of $X_i$ that is smooth is $\le 1+h^1(X,2D)$.
\end{coro}

Another corollary of \thref{main} is a partial converse of the Kodaira vanishing theorem for threefolds of Fano type. Kodaira vanishing theorem tells us that if $X$ is of Fano type and $D$ is a nef Cartier divisor on $X$, then $H^i(X,\O_X(D))=0$ for $i>0$. In the threefold case, we can get a partial converse of the Kodaira vanishing theorem if $D$ is movable.

\begin{coro}\thlabel{coro}
Let $X$ be a threefold of Fano type, and suppose that $D$ is a movable $\Q$-Cartier Weil divisor on $X$. If $H^1(X,\O_X(mD))=0$ for any sufficiently divisible $m$, then $D$ is nef.
\end{coro}

The proof of \thref{main} closely follows the argument of \cite[The proof of Theorem 1.4]{Kim24}. The proof is divided into four steps. In the first step, we run a $D$-MMP. In the second step, we define 
$$ d_i:=h^1(X_i,2mD_i).$$
Utilizing the techniques developed in \cite{Kim25}, we establish $d_{i-1} \geq d_i$. In the third step, we show that if $d_{i-1} = d_i$, then $mD_i$ cannot be Cartier. Finally, in the fourth step, we conclude the proof of the main theorem by combining the results of the previous steps.

\begin{ack}
The author thanks his advisor, Sung Rak Choi, for his comments, questions, and discussions. He is grateful for his encouragement and support. The author is grateful to the anonymous referee for their thoughtful comments and suggestions. The author is partially supported by the Samsung Science and Technology Foundation under Project
Number SSTF-BA2302-03.
\end{ack}

\section{Preliminaries}
The aim of this section is to introduce the notions and lemmas which will be used in later sections. Let us collect the basic notions.

\begin{itemize}
    \item All varieties are integral schemes of finite type over the complex number field $\C$.
    \item For a normal variety $X$, and a boundary $\Q$-Weil divisor $\Delta$ on $X$, we say that $(X,\Delta)$ is a \emph{pair} if $K_X+\Delta$ is a $\Q$-Cartier divisor.
    \item Let $(X,\Delta)$ be a pair, and $E$ a prime divisor over $X$. Suppose $f:X'\to X$ is a log resolution of $(X,\Delta)$ in which $E$ is an $f$-exceptional divisor, and $\Delta_{X'}$ is a divisor on $X'$ such that
    $$ K_{X'}+\Delta_{X'}=f^*(K_X+\Delta)$$
    holds. The \emph{discrepancy} $a(E;X,\Delta)$ is defined by
    $$ a(E;X,\Delta):=\mathrm{mult}_E(-\Delta_{X'}).$$
    Moreover, $(X,\Delta)$ is \emph{klt} if for any prime divisor $E$ over $X$, $a(E;X,\Delta)>-1$ holds.
\end{itemize}

Let us define the notion of \emph{log Fano pair}. For more details, we recommend referring to \cite{PS09}.

\begin{defi}
We say that a pair $(X,\Delta)$ is a \emph{log Fano pair} if $(X,\Delta)$ is klt and $-(K_X+\Delta)$ is ample. Moreover, for a normal projective variety $X$, $X$ is of \emph{Fano type} if there is an effective $\Q$-Weil divisor $\Delta$ on $X$ such that $(X,\Delta)$ is a log Fano pair.
\end{defi}

\begin{rmk}\thlabel{rmk}
Let $X$ be a normal projective variety. Then $X$ is of Fano type if and only if there is a big and effective $\Q$-Weil divisor $\Delta$ on $X$ such that $(X,\Delta)$ is klt, and $K_X+\Delta\sim_{\Q}0$.
\end{rmk}

Let us define the notion of \emph{movable divisor}.

\begin{defi}
Let $X$ be a normal projective variety, and $D$ an effective $\Q$-Cartier $\Q$-Weil divisor. We say that $D$ is \emph{movable} if there is a positive integer $m$ such that $mD$ is Cartier, and $|mD|$ has no fixed component.
\end{defi}

The following lemma explains why we impose the movability condition on \thref{main}.

\begin{lem}\thlabel{lemlem}
Let $X$ be a normal projective variety, and $D$ an effective $\Q$-Cartier $\Q$-Weil divisor. If $D$ is movable, then any $D$-MMP consists only of flips.
\end{lem}

\begin{proof}
Let
$$ X_0:=X\overset{f_1}{\dashrightarrow} X_1\overset{f_2}{\dashrightarrow} \cdots \overset{f_n}{\dashrightarrow} X_n$$
be a $D$-MMP, and suppose that $f_i$ is a divisorial contraction for some $1\le i\le n$. Let us write $D_k$ the strict transform of $D$ on $X_k$ for $1\le k\le n$. We may assume that for all $1\le j<i$, $f_j$ are flips. Then by the fact that
$$ H^0(X,\O_X(mD))=H^0(X_{i-1},\O_{X_{i-1}}(mD_{i-1}))$$
for any positive integer $m$, $D_{i-1}$ is movable. Moreover, by the Negativity lemma, there is a non-zero effective $f_i$-exceptional divisor $E$ such that
$$ D_{i-1}=f^*_iD_i+E$$
holds. Then for any positive integer $m$ such that $mD_i$ is Cartier, any element of $|mD_{i-1}|$ contains $E$, and thus $D_{i-1}$ is not movable. Contradiction.
\end{proof}

The following is \cite[Lemma 2.7]{Kim24}, and we state the lemma for the convenience of readers. It may be well-known to experts.

\begin{lem}\thlabel{KRss}
Let $X$ be a normal projective variety, $D$ a Cartier divisor on $X$, and $f:X\dashrightarrow X^+$ a $D$-flip. Suppose
$$
\begin{tikzcd}
X\ar[dashed,"f"]{rr}\ar["\varphi"']{rd}& & X^{+} \ar["\varphi^+"]{ld}\\
& Z&
\end{tikzcd}
$$
is the diagram corresponding to $f$. Then there is a common resolution $g:X'\to X$ and $g':X'\to X^+$ such that the diagram
$$
\begin{tikzcd}
 & X'\ar["g"']{ld}\ar["g'"]{rd}&\\
    X\ar[dashed,"f"]{rr}\ar["\varphi"']{rd}& & X^{+} \ar["\varphi^+"]{ld}\\
& Z&
\end{tikzcd}
$$
is commutative.
\end{lem}

Let us recall \cite[Lemma 4.7]{Kim25}.

\begin{definition-lemma} \thlabel{6}
Let $X$ be a normal variety, and suppose that $D$ is a $\Q$-Cartier Weil divisor on $X$. For a resolution $f:X'\to X$, we denote by $L_{X',D}$ a Cartier divisor on $X'$ such that
$$ (f^*\O_{X'}(D))^{\vee\vee}=\O_{X'}(L_{X',D})$$
holds, and $E_{X'}$ an effective $f$-exceptional divisor such that $f^*D\sim_{\Q} L_{X',D}+E_{X',D}$ holds.
\end{definition-lemma}

Note that the choice of $L_{X',D}$ and $E_{X',D}$ may not be unique.

Let us state the vanishing theorem for nef Cartier divisors on a Fano type variety to nef $\Q$-Cartier Weil divisors.

\begin{lem}[{cf. \cite[Corollary 5.7.7]{Fuj17}}]\thlabel{van1}
Let $X$ be a variety of Fano type, $\varphi:X\to Z$ a contraction, and $D$ a nef $\Q$-Cartier Weil divisor on $X$. Then
$$ R^i\varphi_*\O_X(D)=0\text{ for all }i>0.$$
\end{lem}

The following lemma is essentially the same as that of \cite[Theorem 1.5]{Kim25}, and we include the proof for completeness.

\begin{lem}\thlabel{van2}
Let $(X,\Delta)$ be a klt pair, $D$ a $\Q$-Cartier Weil divisor on $X$, and $f:X'\to X$ a log resolution of $(X,\Delta)$. Then
$$ R^if_*\O_{X'}(L_{X',D}+\floor{E_{X',D}})=0\text{ for }i>0.$$
\end{lem}

\begin{proof}
Recall from \thref{6} that:
$$ (f^*\O_X(D))^{\vee\vee}=\O_{X'}(L_{X',D}),\text{     }f^*D\sim_{\Q}L_{X'D}+E_{X',D},$$
where $E_{X',D}$ is an effective $f$-exceptional divisor.

Write
$$K_{X'}=f^*(K_X+\Delta)+F-F',$$
where $F$ is an effective, $f$-exceptional Cartier divisor, and $F'$ is a simple normal crossing divisor with $\floor{F'}=0$. Let $\{E_i\}$ be the prime $f$-exceptional divisors. Define
$$ a_i:=\begin{cases} 1 &\text{ if } \mathrm{mult}_{E_i}(F'-\{E\})<0,\text{ and } \\ 0 & \text{ otherwise. }\end{cases}$$
Set
$$ E':=\sum_i a_iE_i.$$
Then we have
$$ F+L_{X',D}+\floor{E_{X',D}}+E'\sim_{\Q,f}K_{X'}+F'-\{E_{X',D}\}+E',$$
and clearly $\floor{F'-\{E_{X',D}\}+E'}=0$. 

By the relative Kawamata-Viehweg vanishing theorem (cf. \cite[Theorem 3.3.4]{Fuj17}), 
$$R^if_*\O_{X'}(F+L_{X',D}+\floor{E_{X',D}}+E')=0\text{ for }i>0.$$
Next, observe that there is a natural injection 
$$\O_{X'}(L_{X',D}+\floor{E_{X',D}})\to \O_{X'}(F+L_{X',D}+\floor{E_{X',D}}+E').$$
By \cite[Lemma 4.4]{Kim25}, we conclude
$$R^if_*\O_{X'}(L_{X',D}+\floor{E_{X',D}})=0\text{ for all }i>0.$$
\end{proof}

We state \cite[Lemma 7.30]{Kol13} for the convenience of readers.

\begin{lem}[{cf. \cite[Lemma 7.30]{Kol13}}]\thlabel{lemma:>-1}
Let $f:X'\to X$ be a proper birational morphism between normal varieties. Let $D$ be a Weil divisor on $X'$ and assume that $D\sim_{\Q,f}D_h+D_v$ where $D_v$ is $f$-exceptional, $\floor{D_v}=0$ and $D_h$ is effective without exceptional components. Let $B$ be any effective, $f$-exceptional divisor. Then
$$ \O_X(-f_*D)=f_*\O_{X'}(-D)=f_*\O_{X'}(B-D).$$
\end{lem}

\section{Proof of main theorem}

The purpose of this section is to prove \thref{main}, \thref{minicoro}, and \thref{coro}.

\begin{proof}
As mentioned in the introduction, the proof follows the same structure as \cite[The proof of Theorem 1.4]{Kim24}, comprising four main steps.

\medskip

\noindent\textbf{Step 1}. By Lemma \ref{lemlem}, each $f_i$ in the sequence of birational maps is a flip. Consider the flipping diagram:
$$
\begin{tikzcd}
X_{i-1}\ar[dashed]{rr}\ar["\varphi_i"']{rd}& &X_i\ar["\varphi^+_i"]{ld}\\
& Z_i&
\end{tikzcd}
$$
Applying Lemma \ref{KRss}, we obtain resolutions $f_i:X'_i\to X_{i-1}$ and $g_i:X'_i\to X_i$ such that $f_i,g_i,\varphi_i,\varphi^+_i$ fit into the following commutative diagram:
\begin{equation}\label{3'}
\begin{tikzcd}
& X'_i\ar["f_i"']{ld}\ar["g_i"]{rd}& \\
X_{i-1}\ar[dashed]{rr}\ar["\varphi_i"']{rd}& & X_i\ar["\varphi^+_i"]{ld}\\
& Z_i.&
\end{tikzcd}
\end{equation}
We claim that each $X_i$ is of Fano type. Indeed, let $\Delta'$ be a big and effective $\Q$-Weil divisor on $X$ such that $(X,\Delta')$ is klt and $K_X+\Delta'\sim_{\Q}0$. Denote by $\Delta'_i=$ the strict transform of $\Delta'$ on $X_i$. Then, by \cite[Theorem 3.7 (4)]{KM98}, $(X_i,\Delta'_i)$ is klt and $K_{X_i}+\Delta'_i\sim_{\Q}0$. Moreover, since $\Delta'_i$ is big, Remark \ref{rmk} ensures that $X_i$ is of Fano type.
\medskip

\noindent\textbf{Step 2}. For each integer $0\le i\le n$, let us set
$$ d_i:=h^1(X_i,\O_{X_i}(2mD_i)).$$
In this step, we show $d_{i-1}\ge d_i$ for all $1\le i\le n$. 

Consider the following Leray spectral sequence:
$$ \begin{aligned}E^{st}_2=H^s&\left(X_{i},R^tg_{i*}\O_{X'_i}\left(L_{X'_i,2mD_{i-1}}+\floor{E_{X'_i,2mD_{i-1}}}\right)\right)\\
&\implies H^{s+t}\left(X'_i,\O_{X'_i}\left(L_{X'_i,2mD_{i-1}}+\floor{E_{X'_i,2mD_{i-1}}}\right)\right).
\end{aligned}$$
Focus on the term
$$ E^{10}_2=H^1\left(X_{i},g_{i*}\O_{X'_i}\left(L_{X'_i,2mD_{i-1}}+\floor{E_{X'_i,2mD_{i-1}}}\right)\right).$$

By the Negativity lemma, we have
$$
L_{X'_i,2mD_{i-1}}+E_{X'_i,2mD_{i-1}}\sim_{\Q}L_{X'_i,2mD_i}+E_{X'_i,2mD_i}+F,
$$
where $F$ is an effective $g_i$-exceptional divisor on $X'_i$. Moreover,
\begin{equation}\label{4'} 
\begin{aligned}
L_{X'_i,2mD_{i-1}}+&\floor{E_{X'_i,2mD_{i-1}}}
\\ &\sim_{\Q}F-\left(-\left(L_{X'_i,2mD_i}+E_{X'_i,2mD_i}\right)+\{E_{X'_i,2mD_{i-1}}\}\right)
\\ &\sim_{\Q,g_i}F-(0+\{E_{X'_i,2mD_{i-1}}\}).
\end{aligned}
\end{equation}
and
\begin{equation}\label{4''} 
\begin{aligned}
L_{X'_i,2mD_{i-1}}+&\floor{E_{X'_i,2mD_{i-1}}}
\\ &\sim_{\Q}-\left(-\left(L_{X'_i,2mD_{i-1}}+E_{X'_i,2mD_{i-1}}\right)+\{E_{X'_i,2mD_{i-1}}\}\right).
\\ &\sim_{\Q,f_i}-\left(0+\{E_{X'_i,2mD_{i-1}}\}\right)
\end{aligned}
\end{equation}

By Lemma \ref{lemma:>-1}, together with \eqref{4'}, where
\begin{equation} \label{rf1}
D=-\left(L_{X'_i,2mD_{i-1}}+\floor{E_{X'_i,2mD_{i-1}}}\right)+F, D_h=0, D_v=\left\{E_{X'_i,2mD_{i-1}}\right\}, B=F,\end{equation}
it follows that 
$$g_{i*}\O_{X'_i}\left(L_{X'_i,2mD_{i-1}}+\floor{E_{X'_i,2mD_{i-1}}}\right)=\O_{X_{i}}(2mD_{i}).$$
Hence,
\begin{equation}\label{ot1} 
E^{10}_2=H^1(X_{i},\O_{X_{i}}(2mD_{i})).
\end{equation}
By Lemma \ref{van2},
$$ R^sf_{i*}\O_{X'_i}\left(L_{X'_i,2mD_{i-1}}+\floor{E_{X'_i,2mD_{i-1}}}\right)=0\text{ for all }s>0.$$
Moreover, using Lemma \ref{lemma:>-1} and \eqref{4''}, where
\begin{equation}\label{rf2}
D=-\left(L_{X'_i,2mD_{i-1}}+\floor{E_{X'_i,2mD_{i-1}}}\right), D_h=0, D_v=\left\{E_{X'_i,2mD_{i-1}}\right\}, B=0
\end{equation}
we obtain
$$ f_{i*}\O_{X'_i}\left(L_{X'_i,2mD_{i-1}}+\floor{E_{X'_i,2mD_{i-1}}}\right)=\O_{X_{i-1}}(2mD_{i-1}).$$
Thus, from the Leray spectral sequence, we obtain
\begin{equation} \label{ot2}
E^s:=H^{s}\left(X'_i,\O_{X'_i}\left(L_{X'_i,2mD_{i-1}}+\floor{E_{X'_i,2mD_{i-1}}}\right)\right)=H^{s}(X_{i-1},\O_{X_{i-1}}(2mD_{i-1}))
\end{equation}
for $s\ge 0$. In particular, the injection $E^{10}_2\hookrightarrow E^1$ with (\ref{ot1}) and (\ref{ot2}) implies $d_{i-1}\ge d_i$.

\noindent\textbf{Step 3}. In this step, we show that if $d_{i-1}=d_i$, then $mD_{i-1}$ cannot be Cartier. Let $1\le j\le n$. Consider the following Leray spectral sequence
\begin{equation}\label{111}
E^{st}_2=H^s(Z_{j},R^t\varphi_{j*}\O_{X_{j-1}}(2mD_{j-1}))\implies H^{s+t}(X_{j-1},\O_{X_{j-1}}(2mD_{j-1})).
\end{equation}
We observe that
$$ 
\begin{aligned}
\varphi_{j*}\O_{X_{j-1}}(2mD_{j-1})&\overset{(1)}{=}(\varphi_j\circ f_j)_*\O_{X'_j}\left(L_{X'_j,2mD_{j-1}}+\floor{E_{X'_j,2mD_{j-1}}}\right)
\\ &\overset{(2)}{=}(\varphi^+_j\circ g_j)_*\O_{X'_j}\left(L_{X'_j,2mD_{j-1}}+\floor{E_{X'_j,2mD_{j-1}}}\right)
\\ &\overset{(3)}{=}\varphi^+_{j*}\O_{X_j}(2mD_j).
\end{aligned}$$
Here, $(1)$ follows from Lemma \ref{lemma:>-1} and \eqref{4''} (with the setting of (\ref{rf2})), $(2)$ uses the commutative diagram \eqref{3'}, and $(3)$ again applies Lemma \ref{lemma:>-1} and \eqref{4'} (with the setting of (\ref{rf1})). 

Additionally, by the relative Kawamata-Viehweg vanishing theorem (cf. \thref{van1}) and another Leray spectral sequence, we obtain
\begin{equation}\label{pclaim} 
H^s(Z_j,\varphi_{j*}\O_{X_{j-1}}(2mD_{j-1}))=H^s(Z_j,\varphi^+_{j*}\O_{X_j}(2mD_j))=H^s(X_j,\O_{X_j}(2mD_j))
\end{equation}
for every $s\ge 0$. Our claim is:
\begin{equation}\label{claim}
H^2(X_j,\O_{X_j}(2mD_j))=0
\end{equation}

Consider the Leray spectral sequence
$$ {}^{'}E^{st}_2=H^s(Z_{j+1},R^t\varphi_{(j+1)*}\O_{X_j}(2mD_j))\implies H^{s+t}(X_j,\O_{X_j}(2mD_j)).$$
Note that ${}^{'}E^{02}_2=0$ since $R^2\varphi_{(j+1)*}\O_{X_j}(2mD_j)=0$ (cf. \cite[Lemma 02V7]{Stacks}), and ${}^{'}E^{11}_2=0$ because the support of $R^1\varphi_{(j+1)*}\O_{X_j}(2mD_j)$ has dimension $0$. Hence, the edge map
\begin{equation} \label{lclaim}
{}^{'}E^{20}_2\to {}^{'}E^2
\end{equation}
is surjective.

From this, we deduce:
$$
\begin{aligned}
H^2(X_{j+1},\O_{X_{j+1}}(2mD_{j+1}))=0 &\overset{(1)}{\iff} H^2(Z_{j+1},\varphi_{(j+1)*}\O_{X_j}(2mD_j))=0
\\ &\overset{(2)}{\implies} H^2(X_j,\O_{X_j}(2mD_j))=0,
\end{aligned}
$$
where (1) follows from \eqref{pclaim} and (2) arises from the surjectivity of \eqref{lclaim}. Thus, 
$$H^2(X_n,\O_{X_n}(2mD_n))=0$$
implies \eqref{claim}, and the vanishing $H^2(X_n,\O_{X_n}(2mD_n))=0$ follows from \thref{van1}.

The first four terms of the five-term exact sequence for \eqref{111} is
$$
\begin{aligned}
    0\to H^1(Z_i,\varphi_{i*}\O_{X_{i-1}}&(2mD_{i-1}))\overset{\alpha}{\to} H^1(X_{i-1},\O_{X_{i-1}}(2mD_{i-1}))
    \\ &\to H^0(Z_{i},R^1\varphi_{i*}\O_{X_{i-1}}(2mD_{i-1}))\to H^2(Z_{i},\varphi_{i*}\O_{X_{i-1}}(2mD_{i-1})).
\end{aligned}$$
Note that the first term is $H^1(X_{i},\mathcal{O}_{X_{i}}(2mD_{i}))$ by \eqref{pclaim}. Together with $d_{i-1}=d_i$, \eqref{pclaim} and \eqref{claim}, we show that $\alpha$ is an isomorphism, and
$$ H^0(Z_i,R^1\varphi_{i*}\O_{X_{i-1}}(2mD_{i-1}))=0.$$
Moreover, since the image of the exceptional locus of $\varphi_{i}$ has dimension $0$, we conclude
$$ R^1\varphi_{i*}\O_{X_{i-1}}(2mD_{i-1})=0.$$

Suppose, for contradiction, that $mD_{i-1}$ is Cartier. Define
$$ \mathcal{L}:=\O_{X_{i-1}}(mD_{i-1})$$
By \cite[Theorem 1]{Kaw91}, there is a rational curve $C\subseteq X_{i-1}$ that is contracted by $\varphi_i$. Consider the exact sequence
\begin{equation} \label{exact}0\to \mathcal{Q}\to \mathcal{L}^{\otimes 2}\to \mathcal{L}^{\otimes 2}|_C\to 0\end{equation}
for some coherent sheaf $\mathcal{Q}$ on $X_{i-1}$. Since every fiber of $\varphi_i$ has dimension $\le 1$, we have $R^2\varphi_{i*}\mathcal{Q}=0$ (cf. \cite[Lemma 02V7]{Stacks}). Thus, applying $\varphi_{i*}$ on \eqref{exact} yields
$$ H^1(C,\mathcal{L}^{\otimes 2}|_C)=R^1\varphi_{i*}\mathcal{L}^{\otimes 2}|_C=0.$$
Let $\nu:\P^1\to C$ be the normalization. From
$$ 0\to \mathcal{L}^{\otimes 2}|_C\to \nu_*\nu^*(\mathcal{L}^{\otimes 2}|_C)\to \mathcal{Q}'\to 0$$
for some coherent sheaf $\mathcal{Q}'$ on $C$, we obtain
$$ H^1(C,\nu_*\nu^*(\mathcal{L}^{\otimes 2}|_C))=0.$$
But $\nu^*(\mathcal{L}^{\otimes 2}|_C)=(\nu^*(\mathcal{L}|_C))^{\otimes 2}$. A Leray spectral sequence plus \cite[Lemma 02OE]{Stacks} implies 
$$
H^1(\P^1,(\nu^*(\mathcal{L}|_C))^{\otimes 2})=0.
$$

Writing $\nu^*(\mathcal{L}|_C)=\O_{\P^1}(-M)$ for some positive integer $M$ means that $\mathcal{L}$ restricts to an anti-ample line bundle on $\P^1$. Hence
$$ H^1(\P^1,\O_{\P^1}(-2M))=0,$$
which contradicts the fact that $H^1(\P^1,\O_{\P^1}(-2M))$ is always nonzero for $M>0$. This contradiction shows $mD_{i-1}$ cannot be Cartier.

\medskip

\noindent\textbf{Step 4}. 
We now complete the proof. By the previous step (Step 3), if $mD_{i-1}$ is Cartier for some $i$, then necessarily $d_{i-1}>d_{i}$. Hence, each time $mD_{i-1}$ is Cartier, the sequence $\{d_i\}$ strictly decreases.

Suppose $N>1+h^1(X,\O_X(2mD))$, where $N$ is the number of $i$ such that $mD_i$ is Cartier. Then, for some $1\le i\le n\le N$, the inequality $d_i\le -1$ would be enforced. This is impossible since each $d_i$ is defined as a cohomological dimension and is therefore non-negative:
$$ d_i:=\dim_{\C}H^1(X_i,\O_{X_i}(2mD_i))\ge 0.$$

\end{proof}

\begin{proof}[Proof of \thref{minicoro}]

Set $m=1$, and let $D_i$ be the strict transform of $D$ on $X_i$. Let $N$ be the number of indices $0\le i\le n$ such that $D_i$ is Cartier, and let $N'$ be the number of $X_i$ that are smooth. Since a Weil divisor on a smooth variety is necessarily Cartier, we have $N'\le N$. Moreover, by Theorem \ref{main}, $N\le 1+h^1(X,2D)$, proving Corollary \ref{minicoro}.
\end{proof}

\begin{proof}[Proof of \thref{coro}]



Let $\Delta$ be an effective $\Q$-Weil divisor on $X$ such that $(X,\Delta)$ is klt and $-(K_X+\Delta)$ is big and nef. By \cite[Theorem 4.1]{Fuj11}, there is a small $\Q$-factorialization $f:(X',\Delta')\to (X,\Delta)$. Since
$$ K_{X'}+\Delta'=f^*(K_X+\Delta)$$
it follows that $-(K_{X'}+\Delta')$ is also big and nef. In particular, $X'$ is of Fano type. Furthermore, $f$ is small, so $f^*D$ is movable.

Let $g:X''\to X'$ be a resolution. Because any klt pair has rational singularities (cf. \cite[Theorem 5.22]{KM98}), we get
$$ R^i(f\circ g)_*\O_{X''}=0,\text{ and }R^ig_*\O_{X''}=0\text{ for }i>0.$$
Then, by the Grothendieck spectral sequence
$$ E^{st}_2=R^sf_*R^tg_*\O_{X''}\implies R^{s+t}(f\circ g)_*\O_{X''},$$
we deduce $R^sf_*\O_{X'}=R^s(f\circ g)_*\O_{X''}=0$ for $s>0$. Hence, by the Leray spectral sequence, $H^1(X',\O_{X'}(f^*(mD)))=0$ for any sufficiently divisible $m$.

Note that $X'$ is a Mori dream space (cf. \cite[Corollary 1.3.1]{BCH+10}). Therefore there is a $D$-MMP
$$ X_0:=X'\dashrightarrow X_1\dashrightarrow \cdots \dashrightarrow X_n.$$
Let $D_i$ be the strict transform of $f^*D$ on $X_i$. Choose an integer $m>0$ so that each $mD_i$ is Cartier. We may assume $h^1(X_0,2mD_0)=0$. By \thref{main}, it follows $n+1\le 1$, i.e., $n=0$. Hence the only possibility is that $D_0$ is nef, which implies $f^*D$ is nef on $X'$. Since $X'$ is a Mori dream space, $f^*D$ is semiample, and thus $D$ itself is semiample. This completes the proof.
\end{proof}

\begin{rmk}\thlabel{rmk1}

In \thref{coro}, the hypothesis that $D$ is movable is essential. For example, let $f:X\to \P^3$ be a blowup of $\P^3$ at a single point, let $H$ be an ample Cartier divisor on $\P^3$, and let $E$ be the effective $f$-exceptional divisor on $X$. Then for any positive integer $n$, we obtain
$$ H^1(\P^3,f_*\O_X(n(f^*H+E)))=H^1(\P^3,\O_{\P^3}(nH))=0,$$
and
$$ H^0(\P^3,R^1f_*\O_X(nf^*H+nE))=0.$$
Note also that $R^1f_*\O_X(nE)=0$. Indeed, $E\cong \P^2$, and we have an exact sequence
$$ 0\to \O_X((n-1)E)\to \O_X(nE)\to \O_{\P^2}(-n)\to 0.$$
Pushing forward via $f$ induces a surjection
$$ R^1f_*\O_X((n-1)E)\to R^1f_*\O_X(nE)\to 0.$$
Becuase $R^1f_*\O_X=0$ (cf. \cite[Theorem 5.22]{KM98}), an induction on $n$ shows $R^1f_*\O_X(nE)=0$. Thus, by the Leray spectral sequence,
$$H^1(X,\O_X(n(f^*H+E)))=0\text{ for any }n>0.$$
Nevertheless, $f^*H+E$ fails to be nef. Consequently, the movability assumption on $D$ is Corollary \ref{coro} is indeed necessary.
\end{rmk}

\end{document}